%% file: quotients_arxiv.tex
\documentclass[a4paper,reqno]{amsart}
\pdfoutput=1

\usepackage{etoolbox}
\newtoggle{arxiv}
\toggletrue{arxiv}

\usepackage[utf8]{inputenc}
\usepackage{amsmath,amssymb,amsfonts,amsthm}
\usepackage{tikz}
\usetikzlibrary{arrows}
\usepackage[normalem]{ulem}
\usepackage{enumerate}
\usepackage{wrapfig}
\usepackage{todonotes}
\usepackage{url}
\usepackage[unicode=true,hidelinks,bookmarks=false]{hyperref}
\usepackage[capitalise]{cleveref}

\usepackage{enumerate}
\usepackage[shortlabels]{enumitem}

\usepackage[yyyymmdd,hhmmss]{datetime}

\usepackage{mathtools}
\usepackage{stmaryrd}
\usepackage{thmtools}
\usepackage{float} 

 \usepackage{relsize}

\theoremstyle{plain}
\newtheorem{theorem}{Theorem}
\newtheorem{corollary}[theorem]{Corollary}
\newtheorem{lemma}[theorem]{Lemma}

\theoremstyle{definition}
\newtheorem{remark}[theorem]{Remark}
\newtheorem{example}[theorem]{Example}
\newtheorem{definition}[theorem]{Definition}
\newtheorem{convention}[theorem]{Convention}

\newcommand{\UU}{\mathcal{U}}

\newcommand{\refl}{\mathsf{refl}}

\newcommand{\ct}{%
  \mathchoice{\mathbin{\raisebox{0.5ex}{$\displaystyle\centerdot$}}}%
             {\mathbin{\raisebox{0.5ex}{$\centerdot$}}}%
             {\mathbin{\raisebox{0.25ex}{$\scriptstyle\,\centerdot\,$}}}%
             {\mathbin{\raisebox{0.1ex}{$\scriptscriptstyle\,\centerdot\,$}}}
}

\newcommand{\trunc}[2]{\mathopen{}\left\Vert #2\right\Vert_{#1}\mathclose{}}

\newcommand{\tproj}[3][]{\mathopen{}\left|#3\right|_{#2}^{#1}\mathclose{}}

\newcommand{\defeq}{\vcentcolon\equiv}

\newcommand{\glue}{\ensuremath{\mathsf{glue}}}

\newcommand{\inl}{\ensuremath{\mathsf{inl}}}
\newcommand{\inr}{\ensuremath{\mathsf{inr}}}

\newcommand{\bool}{\mathbf 2}
\newcommand{\N}{\mathbb N}

\newcommand{\List}{\mathsf{List}}
\newcommand{\unit}{\mathbf 1}

\newcommand{\quot}{A \! \sslash \! {\scriptstyle \sim}}
\newcommand{\squot}{A \slash {\scriptstyle \sim}}

\newcommand{\ap}{\ensuremath{\mathsf{ap}}}

\title[Coherence via Well-foundedness:]{Coherence via Well-foundedness: \\
Taming Set-Quotients in Homotopy Type Theory}\thanks{Funding Acknowledgment: This work was supported by The Royal Society, grant reference URF\textbackslash R1\textbackslash 191055.\\
\indent
This version of the paper is essentially identical to the final publication at \emph{Logic in Computer Science 2020 (LICS'20)}, \url{https://doi.org/10.1145/3373718.3394800}.
}
\author{Nicolai Kraus \and Jakob von Raumer}

\begin{document}

\begin{abstract}
\input{quotients_abstract}
\end{abstract}

\maketitle

\input{quotients_body}

\subsection*{Acknowledgments}
\input{quotients_acks}

\bibliographystyle{plain}
\bibliography{master}

\end{document}

%% file: quotients_abstract.tex
Suppose we are given a graph and want to show a property for all its \emph{cycles} (closed chains).
Induction on the length of cycles does not work since sub-chains of a cycle are not necessarily closed.
This paper derives a principle reminiscent of induction for cycles for the case that the graph is given as the symmetric closure of a locally confluent and (co-)well-founded relation.
We show that, assuming the property in question is sufficiently nice, it is enough to prove it for the empty cycle and for cycles given by local confluence.

Our motivation and application is in the field of \emph{homotopy type theory}, which allows us to work with the higher-dimensional structures that appear in homotopy theory and in higher category theory, making \emph{coherence} a central issue.
This is in particular true for quotienting -- a natural operation which gives a new type for any binary relation on a type and, in order to be well-behaved, cuts off higher structure (\emph{set-truncates}). The latter makes it hard to characterise the type of maps from a quotient into a higher type, and several open problems stem from this difficulty.

We prove our theorem on cycles in a type-theoretic setting and use it to show coherence conditions necessary to eliminate from set-quotients into 1-types, deriving approximations to open problems on free groups and pushouts. We have formalised the main result in the proof assistant Lean.

%% file: quotients_body.tex
\newcommand{\dirrel}{\leadsto}
\newcommand{\consPrime}{\colon \!\!\! \colon \!}
\newcommand{\cons}{\mathop{\consPrime}}
\newcommand{\consLong}{\mathsf{snoc}}

\newcommand{\rotate}{\mathsf{rot}}

\newcommand{\hcolim}{\mathsf{hcolim}}



\DeclareFontFamily{U}{wasysmall}{}
\DeclareFontShape{U}{wasysmall}{m}{n}{
  <-5.5> s*[1.4] wasy5
  <5.5-6.5> s*[1.4] wasy6
  <6.5-7.5> s*[1.4] wasy7
  <7.5-8.5> s*[1.4] wasy8
  <8.5-9.5> s*[1.4] wasy9 
  <9.5-> s*[1.4] wasy10
}{}
\newcommand{\biggerhexagon}{\text{\usefont{U}{wasysmall}{m}{n}\symbol{57}}}

\newcommand{\cycle}[1]{\biggerhexagon^{#1}}

\newcommand{\smallerhexagon}{\scalebox{0.5}{\biggerhexagon}}

\newcommand{\toleads}{\mathrel{\reflectbox{$\leadsto$}}}
\newcommand{\toleadsstar}{\; {}^* \!\! \toleads}

\newcommand{\spantocycle}{\mathfrak{L}}

\section{Introduction}

\emph{Homotopy type theory} \cite{hott-book}, commonly known as \emph{HoTT}, 
is a version of constructive dependent type theory inspired by the observation that types carry the structure of (higher) groupoids~\cite{conf/lics/HofmannS94,bg:type-wkom,lumsdaine:phd} and allow homotopical interpretations~\cite{awodeyWarren_HTmodelsOfIT,kap-lum:simplicial-model}.
The assumption of \emph{uniqueness of identity/equality proofs} (UIP), which says that every equality type (written as $a=b$) has at most one element, is rejected; instead, types satisfying this principle are known as \emph{sets} (or \emph{homotopy sets}, \emph{h-sets}).

This leads to new opportunities.
We can find ``higher'' types (i.e.\ ``non-sets''), whose equality structure may be very complicated.
At the same time, the familiar and desirable concept of quotienting \cite{hofmann:thesis} (from now on called \emph{set-quotienting} \cite[Chp 6.10]{hott-book}) becomes an instance of the more general idea of \emph{higher inductive types}.
With the help of these, theories of higher structures can be developed synthetically in type theory.
One difficulty of constructions is that it is usually not sufficient to know that two things are equal, but we need to know \emph{how} they are equal.
Families of equalities need to ``match'' or ``fit together'': they need to be \emph{coherent}.

The guiding motivation for the theory that we develop in this paper is the needed coherence that arises when working with set-quotients simultaneously with higher types.
Recall from \cite[Chp 6.10]{hott-book} that, for a given type $A : \UU$ and a relation\footnote{A \emph{relation} is simply a type family $A \to A \to \UU$ of which we do not assume any properties. In particular, $\sim$ does not need to be an equivalence relation for the quotient to exist. Some authors call such a relation \emph{proof-relevant} to emphasise that it is not necessarily valued in propositions.}
$(\sim) : A \to A \to \UU$, the set-quotient can be implemented as the higher inductive type
\begin{equation} \label{eq:setquotient}
\begin{aligned}
 & \text{inductive } \squot \text{ where} \\
 &\qquad \iota : A \to \squot \\
 &\qquad \glue : \Pi\{a,b : A\}. (a \sim b) \to \iota(a) = \iota(b) \\
 &\qquad \mathsf{trunc} : \Pi\{x,y : \squot\}. \Pi(p,q : x=y). p=q
\end{aligned}
\end{equation}
From this representation, we can derive the elimination rule:
In order to get a function $f : ({\squot}) \to X$, we need to give a function $g : A \to X$ such that, whenever $a \sim b$, we have $g(a) = g(b)$.
However, this only works if $X$ is a set itself.
If it is not, we have a priori no way of constructing the function $f$.

Let us look at one instance of the problem. We consider the following set-quotient, which we will use as a running example.
It is a standard construction that has been discussed in \cite[Chp 6.11]{hott-book}.
\begin{example}[free group] \label{ex:fg}
 Let $A$ be a set.
 We construct the free group on $A$ as a set-quotient.
 We consider lists over $A + A$, where we think of the left copy of $A$ as positive and the right copy as negative elements.
 For $x : A + A$, we write $x^{-1}$ for the ``inverted'' element: 
 \begin{equation}
  \inl(a)^{-1} \defeq \inr(a) \qquad \qquad \qquad \inr(a)^{-1} \defeq \inl(a)
 \end{equation}
 The binary relation $\sim$ on $\List(A+A)$ is generated by
 \begin{alignat}{3} \label{eq:fg-relation}
  & [\ldots, x_1, x_2, x_2^{-1}, x_3, \ldots] & \; &\sim & \;\; & [\ldots, x_1, x_3, \ldots]. 
 \end{alignat}
 Then, the set-quotient $\List(A+A) \slash \sim$ is the free group on $A$: It satisfies the correct universal property by \cite[Thm 6.11.7]{hott-book}.
\end{example}

Another way to construct the free group on $A$ is to re-use the natural groupoid structure that every type carries; this can be seen as a typical ``HoTT-style'' construction.
It works as follows.
The \emph{wedge of $A$-many circles} is the (homotopy) coequaliser of two copies of the map $A$ into the unit type, 
$\hcolim (A \rightrightarrows \unit)$.
Using a higher inductive type, it can be explicitly constructed:
 \begin{equation} \label{wedge}
 \begin{aligned}
  & \text{inductive } \hcolim (A \rightrightarrows \unit) : \UU \text{ where} \\
  & \qquad \mathsf{base}: \hcolim (A \rightrightarrows \unit) \\
  & \qquad \mathsf{loop}: A \to \mathsf{base} = \mathsf{base}
 \end{aligned}
\end{equation}
Its loop space $\Omega(\hcolim (A \rightrightarrows \unit))$ is by definition simply $\mathsf{base} = \mathsf{base}$.
This loop space carries the structure of a group in the obvious way: the neutral element is given by reflexivity, multiplication is given by path composition, symmetry by path reversal, and every $a:A$ gives rise to a group element $\mathsf{loop}(a)$.
This construction is completely reasonable without the assumption that $A$ is a set, and it therefore defines the free \emph{higher} group (cf. \cite{Kraus:free}); from now on, we write $F_A$ for it:
\begin{equation} \label{eq:FA-definition}
 F_A \defeq  \Omega(\hcolim (A \rightrightarrows \unit)) \text{.}
\end{equation}
In contrast to this observation, the set-quotient of \cref{ex:fg} 
ignores any existing higher structure (cf. \cite[Rem 6.11.8]{hott-book}) and thus really only defines the free ``ordinary'' group.
If we do start with a set $A$, it is a natural question whether the free higher group and the free group coincide:
There is a canonical function 
\begin{equation}\label{eq:higher-to-ordinary-fg}
 F_A \to (\List(A+A) \slash \sim),
\end{equation}
defined analogously to $\Omega(\mathsf{S}^1) \to \mathbb Z$, cf.\ \cite{hott-book}.
Classically, this function is an equivalence.
Constructively, it is an open problem to construct an inverse of \eqref{eq:higher-to-ordinary-fg}.

The difficulties do not stem from the first two constructors of the set-quotient.
Indeed, we have a canonical map
\begin{equation} \label{eq:omega1}
 \omega_1 : \List(A+A) \to F_A
\end{equation}
which maps a list such as $[\inl(a_1), \inr(a_2), \inl(a_3)]$ to the path composition $\mathsf{loop}(a_1) \ct (\mathsf{loop}(a_2))^{-1} \ct \mathsf{loop}(a_3)$.
For this map, we also have
\begin{equation} \label{eq:omega2}
 \omega_2 : \Pi(\ell_1,\ell_2 : \List(A+A)). (\ell_1 \sim \ell_2) \to \omega_1(\ell_1) = \omega_1(\ell_2)
\end{equation}
since consecutive inverse loops cancel each other out.
Therefore, if we define $(\quot)$ to be the higher inductive type \eqref{eq:setquotient} \emph{without} the constructor $\mathsf{trunc}$, i.e.\ the \emph{untruncated quotient} or \emph{coequaliser}, then there is a canonical map
\begin{equation} \label{eq:omega-complete}
 \omega : (\List(A+A) \sslash \sim) \to F_A.
\end{equation}
Thus, the difficulty with defining an inverse of \eqref{eq:higher-to-ordinary-fg} lies solely in the question whether $F_A$ is a set.
This is an open problem which has frequently been discussed in the HoTT community (a slight variation is recorded in \cite[Ex 8.2]{hott-book}).
It is well-known in the community how to circumvent the problem if $A$ has decidable equality.
However, the only piece of progress on the general question that we are aware of is the result in \cite{Kraus:free}, where it is shown that all fundamental groups \cite[Chp 6.11]{hott-book} are trivial.
In other words: Instead of showing that \emph{everything} above truncation level $0$ is trivial, the result shows that a single level is trivial.
The proof in \cite{Kraus:free} uses a rather intricate construction which is precisely tailored to the situation.

The construction of functions $({\squot}) \to X$ is the guiding motivation for the results that we present in this paper.
We do not allow an arbitrary type $X$, however; instead, we assume that $X$ is $1$-truncated.
As an application, we will give a new proof for the theorem that the fundamental groups of $F_A$ are trivial.
We will also show a family of similar statements, by proving a common generalisation.

The characterisation of the equality types of $({\squot})$ makes it necessary to consider \emph{cycles}, or \emph{closed zig-zags}, in $A$.
A cycle is simply an element of the symmetric-transitive closure, for example:

\newdimen\WF
\newdimen\WS

\iftoggle{arxiv}{
\WF=.475\textwidth
\WS=.475\textwidth
\begin{minipage}[t]{\WF}
 \begin{align*}
  & s : a \sim b \qquad p : d \sim c \\
  & t : c \sim b \qquad q : a \sim d
 \end{align*}
\end{minipage}%
}{
\WF=.2\textwidth
\WS=.25\textwidth
\begin{minipage}[t]{\WF}
 \begin{align*}
  & s : a \sim b \\
  & t : c \sim b \\
  & p : d \sim c \\
  & q : a \sim d
 \end{align*}
\end{minipage}%
}
%
\begin{minipage}[t]{\WS}
\begin{equation} \label{eq:cyclepicture}
\begin{tikzpicture}[x=1.5cm,y=-1.50cm,baseline=(current bounding box.center)]
  \tikzset{arrow/.style={shorten >=0.1cm,shorten <=.1cm,-latex}}
\node (A) at (0,0) {$a$}; 
\node (B) at (0,1) {$b$}; 
\node (C) at (1,1) {$c$}; 
\node (D) at (1,0) {$d$}; 

\draw[arrow] (A) to node [left] {$s$} (B);
\draw[arrow] (C) to node [below] {$t$} (B);
\draw[arrow] (D) to node [right] {$p$} (C);
\draw[arrow] (A) to node [above] {$q$} (D);
\end{tikzpicture}
\end{equation}
\end{minipage}
Our first new result (\cref{thm:gensetquotnew}) says:
We get a function $({\squot}) \to X$ into a 1-type $X$ if we have $f : A \to X$ and $h : (a \sim b) \to f(a) = f(b)$, together with the coherence condition stating that $h$ maps any cycle to a commuting cycle in $X$. In the case of the example \eqref{eq:cyclepicture} above, this means that the composition $h(s) \ct h(t)^{-1} \ct h(p)^{-1} \ct h(q)^{-1}$ equals $\refl_{f(a)}$.
\cref{thm:gensetquotnew} is fairly simple, and we do not consider it a major contribution of this paper.
The actual contribution of the paper is to make \cref{thm:gensetquotnew} usable
since, on its own, it is virtually impossible to apply in any non-trivial situation.
The reason for this is that the coherence condition talks about cycles, i.e.\ \emph{closed} zig-zags.
Zig-zags are inductively generated (they are simply a chain of segments), but closed zig-zags are not.
If we have a property on cycles, which we cannot generalise to arbitrary zig-zags, then there is no obvious inductive strategy to show the property for all cycles: if we remove a segment of a cycle, the remaining zig-zag is not closed any more.
In all our examples, it seems not possible to formulate an induction hypothesis based on not-necessarily-closed zig-zags.

Although the function space $({\squot}) \to X$ \emph{in homotopy type theory} is the guiding motivation of this paper, the actual main contribution and heart of the paper is completely independent of type theory and could be formulated in almost any (constructive) foundation.
We start from a Noetherian (co-well-founded) and locally confluent binary relation $\leadsto$ on $A$ (conditions that are satisfied in all our examples).
We then construct a new Noetherian binary relation $\leadsto^{\smallerhexagon}$ on the cycles of $\leadsto$.
An important property of this new relation is that any cycle can be written as the ``merge'' of a smaller cycle and one which is given by the condition of local confluence.
This can be seen as a strengthening of Newman's Lemma, and vaguely corresponds to an observation in Neuman's original work that cycles can be decomposed using local confluence \cite{newman1942theories}.
The concrete consequence is \cref{thm:cycle-ind}, which we call \emph{Noetherian cycle induction}:
Assume we are given a property $P$ on cycles which respects merging cycles as well as ``rotating'' cycles.
Then, we can show $P$ for all cycles by showing it for the empty cycle(s) and for all cycles that are given by the local confluence property.

While it is very hard to show a property directly for \emph{all} cycles, it is much more manageable to show the property for those cycles that stem from local confluence.
In any given example, when we prove local confluence, we can \emph{choose} how these ``confluence cycles'' look, giving us full control over what we have to prove.

Let us get back to homotopy type theory.
The combination of the two mentioned results (\cref{thm:gensetquotnew} and \cref{thm:cycle-ind}) gives us \cref{thm:dirsetquot}: Given a 1-type $X$ and $f : A \to X$ such that $a \sim b$ implies $f(a) = f(b)$, it suffices to show that ``confluence cycles'' are mapped to trivial equality proofs.
We apply this to show that the free higher group over a set has trivial fundamental groups.
There is a family of similar statements that we also discuss and prove.

\paragraph*{Outline of the paper.}

We start by clarifying the setting and recalling previous results that we build upon in \cref{sec:prelims}.
We then show the mentioned \cref{thm:gensetquotnew} describing functions out of the set-quotient into $1$-types in \cref{sec:elim-out-of-quots}.
In \cref{sec:red-schemes}, we identify the properties of the relations in question, namely being Noetherian and local confluence.
\cref{sec:order-cycles} is the core of the paper: Starting from a Noetherian locally confluent relation on $A$, we construct a Noetherian relation on the cycles of $A$.
This allows us to formulate the principle of Noetherian cycle induction in \cref{sec:fam-of-cycles}.
In \cref{sec:apps}, we show several applications.
\cref{sec:concl} is reserved for some concluding remarks.

\paragraph*{Formalisation.}

We have formalised the complete construction with all proofs in \cref{sec:prelims,sec:elim-out-of-quots,sec:red-schemes,sec:order-cycles,sec:fam-of-cycles} in the proof assistant Lean~\cite{moura:lean},
including the existing results that we recall in \cref{sec:prelims}.
Parts of \cref{sec:apps} are not self-contained and have not been included in the formalisation, but the important
\cref{thm:dirsetquot} is formalised.
The formalisation was done for the third major version of Lean, which supports
reasoning in HoTT by enforcing that its strict, impredicative universe of propositions
is avoided in all definitions.
It relies on a port of the Lean HoTT library~\cite{vandoorn:lean} which was developed for Lean 2.
The Lean code is available online.\footnote%
{The online repository can be found at \url{https://gitlab.com/fplab/freealgstr}.}

The formalisation approximately follows the structure of the paper.
Many arguments are directly translated from the informal mathematical style of the paper into Lean code.
In some cases, the formalisation uses shortcuts compared to the paper.
These shortcuts would look ad-hoc and unnecessary on paper, where certain overhead is invisible anyway.
However, avoiding this overhead simplifies the formalisation.

While homotopy type theory often allows very short formalisations of ``homotopical'' arguments (such as \cref{thm:gensetquotnew}), combinatorial arguments (such as the ones leading to \cref{thm:cycle-ind}) are more tedious in Lean than they are on paper.
It is difficult to say how big the formalisation exactly is due to the use of existing results, some of which are available in libraries while others are not.
Only counting the core results, the Lean formalisation comprises approximately 1.600 lines of code.

\section{Preliminaries and Prerequisites} \label{sec:prelims}

\subsection{Setting}

We work in the version of homotopy type theory that is developed in the ``HoTT book''~\cite{hott-book}.
This means that we use a dependent type theory with $\Sigma$- and $\Pi$-types as well as coproducts, with univalent universes and function extensionality.
We use inductive families \cite{dybjer1994inductive} as for example in \eqref{eq:refl-trans-rel-def} below.%
\footnote{It is known (cf.\ \cite{ambrusJakob:families}) that any development involving inductive families can be expressed with only ordinary inductive types.}
Further, we use curly brackets to denote implicit arguments as it is done in various proof assistants, for example as in $f : \Pi\{a:A\}. (B(a) \to C(a))$.
This allows us to omit the implicit argument and write $f(b) : C(b)$ in situations where $a$ can easily be inferred. 
The only higher inductive types that are relevant for this paper besides set-quotients are coequalisers
and set-truncations. They are briefly introduced in \cref{subsec:hits} below.
(For the advanced application given by \cref{thm:pushout-1-type}, pushouts \cite[Chp 6.8]{hott-book} are used.)

\subsection{Relations and Closures} \label{subsec:relations}

As usual, we have the following closure operations:
\begin{definition}[closures] \label{def:closures}
 Given a relation $(\sim) : A \to A \to \UU$, we write:
 \begin{itemize}
  \item $\sim^+$ for its transitive closure,
  \item $\sim^*$ for its reflexive-transitive closure,
  \item $\sim^s$ for its symmetric closure,
  \item $\sim^{s*}$ for its reflexive-symmetric-transitive closure.
 \end{itemize}
\end{definition}

The type-theoretic implementation of these closures is standard.
The reflexive-transitive closure is constructed as an inductive family with two constructors:
\begin{equation} \label{eq:refl-trans-rel-def}
 \begin{aligned}
  & \text{inductive } (\sim^{*}) : \, A \to A \to \UU \text{ where} \\
  & \qquad \mathsf{nil}: \Pi \{a:A\}. (a \sim^* a) \\
  & \qquad \consLong: \Pi\{a,b,c:A\}. (a \sim^* b) \to (b \sim c) \to (a \sim^* c)
 \end{aligned}
\end{equation}
The symmetric closure is the obvious disjoint sum,
\begin{equation} \label{eq:sym-rel-def}
 (a \sim^s b) \defeq (a \sim b) + (b \sim a).
\end{equation}
The transitive closure can be defined analogously to \eqref{eq:refl-trans-rel-def}, with the constructor $\mathsf{nil}$ replaced by a constructor of the form $(a \sim b) \to (a \sim^+ b)$.
Note that the relation $\sim$ is not required to be valued in propositions or even sets and neither are any of its closures.
As a consequence, we have functions $(a \sim^* b) \to (a \sim^{**} b)$ and $(a \sim^{**} b) \to (a \sim^* b)$, but we do in general not have $(a \sim^* b) = (a \sim^{**} b)$.
The analogous caveat holds for the other closure operations.

\begin{remark} \label{rem:nest-relation-constructions}
 We often want to nest several constructions.
 If we start with the relation $\sim$ and first apply construction $x$, then $y$, then $z$, we denote the resulting relation by $\sim^{xyz}$.
 For example, the reflexive-symmetric-transitive closure $\sim^{s*}$ is constructed by first taking the symmetric and then the reflexive-transitive closure.

 Needless to say, there are several possible variations of how such concepts can be implemented; for example, $\sim^*$ could be defined as $(x \sim^+ y) + (x = y)$. 
 The definitions we give are the ones which we found lead to the best computational behaviour in a proof assistant.
\end{remark}

\begin{definition}[chains and cycles] \label{def:chains-cycles}
 Given a relation $\sim$, we refer to elements of $(a \sim^{s*} b)$ as \emph{chains (from $a$ to $b$)}.
 A chain is \emph{monotone} if either each segment comes from the left summand in \eqref{eq:sym-rel-def}, i.e.\ is of the form $\mathsf{inl}(t)$ with $t : a \sim b$ (no segment is inverted), or if segment comes from the right summand (every segment is inverted).

 An element of $(a \sim^{s*} a)$ is called a \emph{cycle (based at $a$)}.
 We write $\cycle \sim$ for the type of cycles,
 \begin{equation}
  \cycle \sim \defeq \Sigma(a : A). (a \sim^{s*} a).
 \end{equation}
\end{definition}

\cref{fig:chains-and-cycles} below illustrates these definitions, where $a \to b$ denotes $a \sim b$.
What we call \emph{chain} is a $\sim$-\emph{conversion} from the rewriting perspective \cite{CR1936}, and \emph{cycles} are similar to the \emph{diagrams} of \cite{vincent-16}. 

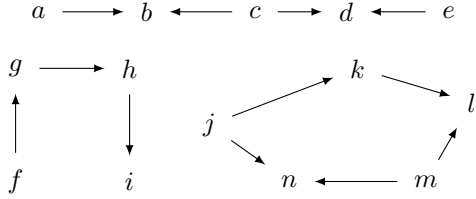
\begin{figure}[h] 
\begin{center}
\begin{tikzpicture}[x=1.5cm,y=-1.50cm,baseline=(current bounding box.center)]
  \tikzset{arrow/.style={shorten >=0.1cm,shorten <=.1cm,-latex}}

\node (P1) at (0.2,0) {$a$}; 
\node (P2) at (1.15,0) {$b$}; 
\node (P3) at (2.1,0) {$c$}; 
\node (P4) at (2.9,0) {$d$}; 
\node (P5) at (3.8,0) {$e$}; 

\draw[arrow] (P1) to node [below] {} (P2);
\draw[arrow] (P3) to node [below] {} (P2);
\draw[arrow] (P3) to node [left] {} (P4);
\draw[arrow] (P5) to node [left] {} (P4);

\node (Q1) at (0,1.5) {$f$};
\node (Q2) at (0,.5) {$g$};
\node (Q3) at (1,.5) {$h$};
\node (Q4) at (1,1.5) {$i$};

\draw[arrow] (Q1) to node [left] {} (Q2);
\draw[arrow] (Q2) to node [left] {} (Q3);
\draw[arrow] (Q3) to node [left] {} (Q4);

\node (R1) at (1.7,1) {$j$};
\node (R2) at (3,.5) {$k$};
\node (R3) at (4,.8) {$l$};
\node (R4) at (3.6,1.5) {$m$};
\node (R5) at (2.4,1.5) {$n$};

\draw[arrow] (R1) to node [left] {} (R2);
\draw[arrow] (R2) to node [left] {} (R3);
\draw[arrow] (R4) to node [left] {} (R3);
\draw[arrow] (R4) to node [left] {} (R5);
\draw[arrow] (R1) to node [left] {} (R5);
\end{tikzpicture}
\end{center}
\caption{A chain from $a$ to $e$, a monotone chain from $f$ to $i$, and a cycle. We sometimes call the elements of $A$ \emph{vertices}.}\label{fig:chains-and-cycles}
\end{figure}

\begin{definition}[notation and operations on closures]
 We have the following standard operations:
 \begin{enumerate}
  \item \emph{reflexivity:} We write $\epsilon : a \sim^* a$ for the constructor $\mathsf{nil}$.
  We also write $\epsilon$ for the trivial cycle, and $\epsilon_a$ if we want to emphasise that we mean the trivial cycle based at $a$.
  \item \emph{transitivity:} Given $\alpha : a \sim^* b$ and $\beta : b \sim^* c$, we notate their concatenation by $(\alpha \cons \beta) : a \sim^* c$. Overloading the symbol $\cons$, we also use it instead of the constructor $\consLong$.
  \item \emph{symmetry:} We can invert $\alpha : a \sim^{s*} b$ to get $\alpha^{-1} : b \sim^{s*} a$.
  This is done by inverting each single segment and reversing the order of all segments.
 \end{enumerate}
\end{definition}

\begin{definition}[empty chains and cycles]
 There is an obvious function
 \begin{equation}
  \mathsf{length}: \Pi\{a,b : A\}. (a \sim^{s*} b) \to \N
 \end{equation}
 which calculates the length of a chain.
 We say that a chain $\alpha$ is \emph{empty} if its length is zero.
 For the type of \emph{empty cycles}, we write
 \begin{equation}
  \biggerhexagon^\sim_\emptyset \defeq \Sigma(\alpha : \cycle \sim). \mathsf{length}(\alpha) = 0.
 \end{equation}
\end{definition}

\begin{remark} \label{rem:empty-versus-trivial}
 The trivial cycle $\epsilon$ is always empty, but not every empty cycle is equal to $\epsilon$.
 Instead, an empty cycle corresponds to a loop in $A$,
 \begin{equation}
  (\Sigma(a:A).(a = a)) \simeq \biggerhexagon^\sim_\emptyset.
 \end{equation}
 This is easy to see, since \eqref{eq:refl-trans-rel-def} without the second constructor is the usual definition of Martin-L\"of's identity type as an inductive family.
 In particular, if $A$ is a set, then $\epsilon$ \emph{is} the only empty cycle.
\end{remark}

\subsection{Quotients, Coequalisers, Truncations} \label{subsec:hits}

As explained in the introduction, the \emph{set-quotient} ${\squot}$ is the higher inductive type with constructors $\iota$, $\glue$, and $\mathsf{trunc}$, see \eqref{eq:setquotient}.
The construction can be split into two steps.
We write $({\quot})$ for the \emph{untruncated quotient} or \emph{coequaliser}
which has only the constructors $\iota$ and $\glue$,
and we write $\trunc 0 -$ for the \emph{set-truncation} which has only the constructors $\iota$ and $\mathsf{trunc}$.
\begin{lemma} \label{lem:simu-is-consec}
 For a relation $\sim$ on $A$, we have
 \begin{equation}
  ({\squot}) \simeq \trunc 0 {\quot}.
 \end{equation}
\end{lemma}
\begin{proof}
 The direct approach of constructing functions back and forth works without difficulties.
\end{proof}
For a given type $X$, there is a canonical function from the function type $(\quot) \to X$ to the $\Sigma$-type of pairs $(f,h)$, where
\begin{align}
 & f : A \to X \label{eq:fAX}\text{,}\\
 & h : \Pi\{a, b : A\}. (a \sim b) \to f(a) = f(b). \label{eq:hAX}
\end{align}
This map is given by:
\begin{equation} \label{eq:canonicalmap}
 g \mapsto (g \circ [-], \ap_g \circ \glue).
\end{equation}
The universal property of the higher inductive type $\quot$ tells us that this function is an equivalence (one can of course also show this with the dependent elimination principle of $\quot$, if that is assumed instead as primitive).


\subsection{Path Spaces of Coequalisers}

We will need to prove statements about equalities in coequalisers.
For this, we use the following result which characterises the path spaces of $(\quot)$: 

\begin{theorem}[induction for coequaliser equality,~\cite{KrausVonRaumer_pathSpaces}] 
\label{thm:lics2019-main}
 Let a relation $(\sim) : A \to A \to \UU$ as before and a point $a_0: A$ be given.
 Assume we further have a type family
 \begin{equation} \label{eq:mainresult-based-P}
 P: \Pi \{b : A\}.(\iota(a_0) =_{\quot} \iota(b)) \to \UU
 \end{equation}
 together with terms
  \begin{align}
   & r : P(\refl_{\iota(a_0)})\text{,} \\
   & e : \Pi\{b,c : A\}, (q: \iota(a_0) = \iota(b)), (s : b \sim c). \nonumber \\
   & \phantom{e : \Pi \{ } P(q) \simeq P(q \ct \glue(s))\text{.}
  \end{align}
 Then, we can construct a term
 \begin{equation}
  \mathsf{ind}_{r,e} : \Pi\{b: A\},(q : \iota(a_0) = \iota(b)). P(q)
 \end{equation}
 with the following $\beta$-rules:
 \begin{align}
  & \mathsf{ind}_{r,e} (\refl_{\iota(a_0)}) = r \label{eq:thm-based-first-beta}\text{,} \\
  & \mathsf{ind}_{r,e}(q \ct \glue(s)) = e (q,s, \mathsf{ind}_{r,e}(q))\text{.} \label{eq:thm-based-second-beta}
 \end{align}
 \qed
\end{theorem}

\subsection{Functions out of Set-Truncations}

For types $A$ and $B$, we have a canonical function
\begin{equation} \label{eq:comp-with-tproj}
 (\trunc 0 A \to B) \to (A \to B)
\end{equation}
which is given by precomposition with $\tproj 0 -$.
Any such function $g \circ \tproj 0 -$ is moreover constant on loop spaces in the sense that
\begin{equation}
 \ap_{g \circ \tproj 0 -} : (a = a) \to (g(a) = g(a))
\end{equation}
satisfies $\ap_{g \circ \tproj 0 -}(p) = \refl$, for all $a$ and $p$.
For a $1$-truncated type $B$, the following known result by Capriotti, Kraus, and Vezzosi states that this property is all one needs to reverse \eqref{eq:comp-with-tproj}:

\begin{theorem}[\cite{capKraVez_elimTruncs}] \label{lem:CSL}
 Let $A$ be a type and $B$ be a 1-truncated type.
 The canonical function from $(\trunc 0 A \to B)$ to the type 
 \begin{equation}
  \Sigma(f : A \to B). \Pi(a:A),(p : a=a). \ap_f(p) = \refl
 \end{equation}
 is an equivalence. \qed
\end{theorem}

\section{Eliminating out of Set-Quotients} \label{sec:elim-out-of-quots}

As before, let $\sim$ be a relation on $A$.
Assume further that given are a function $f : A \to X$ and a proof $h$ that $f$ sends related points to equal points, as in \eqref{eq:fAX} and \eqref{eq:hAX}.
There is an obvious function
\begin{equation}
 h^{s*} : \Pi\{a,b : A\}. (a \sim^{s*} b) \to f(a) = f(b),
\end{equation}
defined by recursion on $a \sim^{s*} b$ which in each step path-composes with an equality given by $h$ or the inverse of such an equality.
Given $(f,h)$ and a third map $k : X \to Y$, it is easy to prove by induction on $a \sim^{*} b$ that we have
\begin{equation} \label{eq:star-shifting}
 \ap_k \circ h^{s*} = (\ap_k \circ h)^{s*}.
\end{equation}
We also note that, for chains $\alpha, \beta$,
\begin{align}
 & h^{s*}(\alpha :: \beta) = h^{s*}(\alpha) \ct h^{s*}(\beta)\label{eq:star-inverting} \text{ and}\\
 & h^{s*}(\alpha^{-1}) = (h^{s*}(\alpha))^{-1}\label{eq:star-functoriality}\text{.}
\end{align}
Of particular interest is the function $\glue^{s*} : \Pi\{a,b : A\}. (a \sim^{s*} b) \to \iota(a) = \iota(b)$.
It is in general not an equivalence: For example, for $t : a \sim a$, the chain $t \cons t^{-1}$  and the empty chain both get mapped to $\refl$.
Thus, $\glue^{s*}$ does not preserve inequality (but see \cref{rem:empty-versus-trivial}).
However, we have the following result:
\begin{lemma}
 The function $\glue^{s*} : (a \sim^{s*} b) \to \iota(a) = \iota(b)$ is surjective.
\end{lemma}
\begin{proof}
 Fixing one endpoint $a_0 : A$ and setting
 \begin{align}
  & P : \Pi\{b:A\}.(\iota(a_0) = \iota(b)) \to \UU \\
  & P(q) \defeq \trunc{-1}{\Sigma(p : a_0 \sim^{s*} b).\glue^{s*}(p) = q}
 \end{align}
we need to show that, for all $q$, we have $P(q)$.
We use \cref{thm:lics2019-main}, where $r$ is given by the empty chain.
To construct $e$, we need to prove $P(q) \simeq P(q \ct \glue(s))$ for any $s : b \sim c$.
This amounts to constructing functions in both directions between the types ${\Sigma(p : a_0 \sim^{s*} b).\glue^{s*}(c) = q}$ and ${\Sigma(p : a_0 \sim^{s*} b).\glue^{s*}(c) = q \ct \glue(s)}$, where extending a chain with $s$ or with $s^{-1}$ is sufficient.
\end{proof}
The following is a ``derived induction principle'' for equalities in coequalisers:
\begin{lemma} \label{lem:simple_ind_paths}
 For a family ${P : \Pi\{x : \quot\}. x=x \to \UU}$ such that each $P(q)$ is a proposition,
 the two types
 \begin{equation}
  \Pi(\gamma : \cycle \sim).P(\glue^{s*}(\gamma)).
 \end{equation}
 and
 \begin{equation}
  \Pi(x:\quot),(q: x=x).P(q)
 \end{equation}
 are equivalent.
\end{lemma}
\begin{proof}
 Both types are propositions, and the second clearly implies the first.
 For the other direction, induction on $x$ lets us assume that $x$ is of the form $\iota(a)$ for some $a:A$; the case for the constructor $\glue$ is automatic.
 The statement then follows from the surjectivity of $\glue^{s*}$.
\end{proof}

\begin{theorem} \label{thm:gensetquotnew}
 Let $A : \UU$ be a type, $(\sim) : A \to A \to \UU$ be a relation, and $X : \UU$ be a 1-type.
 Then, the type
 of functions $({\squot} \to X)$ is equivalent to the type of triples $(f,h,c)$ (a nested $\Sigma$-type), where
 \begin{align}
  & f : A \to X \\
  & h : \Pi\{a, b : A\}. (a \sim b) \to f(a) = f(b) \\
  & c : \Pi(\gamma : \cycle \sim). h^{s*}(\gamma) = \refl. \label{eq:genset-last-comp}
 \end{align}
\end{theorem}
\begin{proof}
 We have the following chain of equivalences:
\iftoggle{arxiv}{
\begin{alignat*}{5}
  &&&& \quad & \phantom{\Sigma} \squot \to X \\[.3cm]
  \textit{by } \cref{lem:simu-is-consec} &\quad &&\simeq & \quad & \phantom{\Sigma} \trunc{0}{\quot} \to X \\[.3cm]
  \textit{by } \cref{lem:CSL} &&&\simeq  &       &\Sigma g : (\quot) \to X. \\
  &&& & \quad & \phantom{\Sigma} c : \Pi\{x:\quot\},(q : x=x). \ap_g(q) = \refl \\[.3cm]
  \textit{by }\cref{lem:simple_ind_paths}&&&\simeq  &  & \Sigma g : (\quot) \to X. \\
  &&& & \quad & \phantom{\Sigma} c : \Pi(\gamma \cycle \sim). \ap_g(\glue^{s*}(\gamma)) = \refl \\[.3cm]
  \textit{by }\eqref{eq:star-shifting}&&&\simeq  &  & \Sigma g : (\quot) \to X. \\
  &&& & \quad & \phantom{\Sigma} c : \Pi(\gamma : \cycle \sim). (\ap_g \circ \glue)^{s*}(\gamma) = \refl \\[.3cm]
  \textit{by }\eqref{eq:canonicalmap}&&&\simeq   &  & \Sigma f : A \to X. \\
  &&& & \quad & \Sigma h : \Pi\{a, b : A\}. (a \sim b) \to f(a) = f(b). \\
  &&&&& \phantom{\Sigma} c : \Pi(\gamma : \cycle \sim). h^{s*}(\gamma) = \refl
 \end{alignat*}
}{
\begin{alignat*}{5}
  &&&\quad & \quad & \phantom{\Sigma} \squot \to X \\[.3cm]
  &&& \simeq & \qquad & \textit{by } \cref{lem:simu-is-consec} \\[.3cm]
  &&&\quad & \quad & \phantom{\Sigma} \trunc{0}{\quot} \to X \\[.3cm]
  &&& \simeq & \qquad & \textit{by } \cref{lem:CSL} \\[.3cm]
  &&& &       &\Sigma g : (\quot) \to X. \\
  &&& & \quad & \phantom{\Sigma} c : \Pi\{x:\quot\},(q : x=x). \ap_g(q) = \refl \\[.3cm]
  &&& \simeq & \qquad & \textit{by }\cref{lem:simple_ind_paths} \\[.3cm]
  &&& & \quad & \Sigma g : (\quot) \to X. \\
  &&& & \quad & \phantom{\Sigma} c : \Pi(\gamma \cycle \sim). \ap_g(\glue^{s*}(\gamma)) = \refl \\[.3cm]
  &&& \simeq & \qquad & \textit{by }\eqref{eq:star-shifting} \\[.3cm]
  &&& & \quad & \Sigma g : (\quot) \to X. \\
  &&& & \quad & \phantom{\Sigma} c : \Pi(\gamma : \cycle \sim). (\ap_g \circ \glue)^{s*}(\gamma) = \refl \\[.3cm]
  &&& \simeq & \qquad & \textit{by }\eqref{eq:canonicalmap} \\[.3cm]
  &&&  & \quad & \Sigma f : A \to X. \\
  &&& & \quad & \Sigma h : \Pi\{a, b : A\}. (a \sim b) \to f(a) = f(b). \\
  &&&&& \phantom{\Sigma} c : \Pi(\gamma : \cycle \sim). h^{s*}(\gamma) = \refl
 \end{alignat*}
}
\end{proof}

\section{On Confluence and Well-Foundedness} \label{sec:red-schemes}

\newcommand{\acc}[1]{\mathsf{acc}^{\mathsmaller{#1}}}

In the theory of \emph{rewriting systems}, a usually desirable property of a rewriting relation is \emph{strong normalisation},
meaning that any term can be rewritten to exactly one irreducible term.
The relations that we are interested in are much weaker.
For a rewriting system, it is usually decidable whether and to what a term can be rewritten. 
In contrast, it is in the examples discussed in the introduction generally undecidable whether, for a given $a : A$, there is a $b$ such that $a \sim b$.
Even if we already have both $a$ and $b$, it is undecidable whether $a \sim b$.
Nevertheless, the concepts of \emph{confluence} and \emph{well-foundedness} make sense and are highly useful in our setting to solve cases such as the ones from the introduction.

\begin{convention}
 To emphasise the ``directedness'' of a relation, we name relations $\leadsto$ instead of $\sim$ when talking about confluence.
 This is only to be understood as supporting the intuition, it does not come with any implicit assumptions.
 Similarly, we use the relation name $<$ in the context of well-foundedness.
 We also use the name \emph{order} synonymously with \emph{relation}; note that an order is still simply a type family $A \to A \to \UU$, again without any implicit assumptions.
\end{convention}

\begin{definition}[span]
 Given a relation $\leadsto$, a \emph{span} is a 5-tuple $(a,b,c,s,t)$ of $a,b,c : A$ together with $s : a \leadsto b$ and $t : a \leadsto c$.
 We write $(\cdot \toleads \cdot \leadsto \cdot)$ for the type of spans.
 If $b$ and $c$ are fixed, we write $(b \toleads \cdot \leadsto c)$ for the type of triples $(a,s,t)$.
 An \emph{extended span} is the same 5-tuple as s span, but with $s : a \leadsto^* b$ and $t : a \leadsto^* c$.
 The notion of a \emph{cospan} and an \emph{extended cospan} are defined analogously, with the directions of $s$ and $t$ reversed.
 We use the notations $(\cdot \toleadsstar \cdot \leadsto^* \cdot)$, and $(b \leadsto^* \cdot \toleadsstar c)$ and so on in the obvious way.
\end{definition}

Spans and cospans are also known as \emph{local peaks} and \emph{local valleys}, respectively \cite{CR1936}. Extended spans and cospans are, in their terminology, \emph{(global) peaks} and \emph{valleys}. 

\begin{definition}[confluence] \label{def:confluence}
 We say that a relation $\leadsto$ on $A$ is \emph{locally confluent} if, for any span, there is a matching extended cospan:
 \begin{equation}
    \mathsf{locConf}(\leadsto) \; \defeq \; \Pi \{b,c : A\}.(b \toleads \cdot \leadsto c) \to (b \leadsto^* \cdot \toleadsstar c).
 \end{equation}
 We say that $\leadsto$ is \emph{confluent} if we can replace the assumption $(b \toleads \cdot \leadsto c)$ by the weaker 
 assumption $(b \toleadsstar \cdot \leadsto^* c)$.
\end{definition}

\begin{remark}
 Note that \emph{being [locally] confluent} is not a proposition; it should be understood as \emph{carrying a [local] confluence structure}.
\end{remark}

The definition of \emph{well-foundedness} is standard as well.
If $x < y$, we say that $x$ is smaller than $y$.
Recall that a point $a : A$ is \emph{$<$-accessible} if every $x < a$ is $<$-accessible, and $<$ is \emph{well-founded} if all points are $<$-accessible.
Type-theoretically, this can be expressed as follows:
\begin{definition}[$\Phi_<$ in \cite{aczelinductive}]
 The family $\acc < : A \to \UU$ is generated inductively by a single constructor,
 \begin{equation}
     \mathsf{step}: \Pi(a:A). (\Pi(x:A). (x < a) \to \acc < (x)) \to \acc < (a)
 \end{equation}
 Further, we define
 \begin{equation}
  \mathsf{isWellFounded}(<) \defeq \Pi(a:A). \acc < (a).
 \end{equation}
\end{definition}

While the definition in \cite[Chp.~10.3]{hott-book} is only given for the special case that $A$ is a set and $<$ is valued in propositions, the more general case that we consider works in exactly the same way (cf.\ our formalisation).
In particular, we have the following two results:

\begin{lemma}
 For any $x$, the type $\acc < (x)$ is a proposition.
 Further, the statement that $<$ is well-founded is a proposition. \qed
\end{lemma}

\begin{lemma}[{accessibility induction \cite[Chp.~10.3]{hott-book}}]
 Assume we are given a family $P : A \to \UU$ such that we have
 \begin{equation}
  \Pi(a_0 : A). \acc < (a_0) \to \left(\Pi(a<a_0). P(a)\right) \to P(a_0).
 \end{equation}
 In this case, we get:
 \begin{equation}
  \Pi(a_0 : A). \acc < (a_0) \to P(a_0).
 \end{equation}
 If $<$ is well-founded, the argument $\acc < (a_0)$ can be omitted and the principle is known as \emph{well-founded induction}.
\end{lemma}

An immediate application is the following:
\begin{lemma} \label{lem:a<a-no-acc}
 If $a < a$, then $a$ is not $<$-accessible.
\end{lemma}
\begin{proof}
 $<$-accessibility induction with the family $P(a) \defeq (a < a) \to \emptyset$.
\end{proof}

\begin{lemma} \label{lem:acc-trans-acc}
 If $a:A$ is $<$-accessible, then it is $<^+$-accessible.
\end{lemma}
\begin{proof}
 By $<$-accessibility induction on $P(a) \defeq (\acc < a) \to (\acc {<^+} a)$.
\end{proof}

The notion relevant to our examples and application is not well-foundedness, but co-well-foundedness:
instead of ``no infinite sequences to the left'' the property we want to use is ``no infinite sequences to the right''.
Of course, this amounts to a simple swap of arguments.
Such a relation is sometimes referred to as \emph{Noetherian} in the rewriting literature, or simply as \emph{terminating}.
We avoid the latter terminology since the non-decidable character of the relation is important in our setting.

\begin{definition}[Noetherian]
 A relation $>$ is called \emph{Noetherian} (or \emph{co-well-founded}) if the opposite relation $>^\mathsf{op}$, defined by $(b >^\mathsf{op} a) \defeq (a > b)$, is well-founded.
 The corresponding induction principle is known as \emph{Noetherian induction}.
\end{definition}

We have already seen an example in the introduction of this paper.
Recall the relation \eqref{eq:fg-relation} that is used in the construction of the free group in \cref{ex:fg}:
Given two lists $\ell_1,\ell_2 : \List(A+A)$, we have $\ell_1 \leadsto \ell_2$ if the first list can be transformed into the second list by removing exactly two elements.
The two removed list elements have to be consecutive and ``inverse'' to each other, i.e.\ one is of the form $\inl(a)$, the other $\inr(a)$.

\begin{lemma}[{free groups, continuing \cref{ex:fg}}] \label{lem:fg-Noether-confl}
 The relation $\leadsto$ on lists, in \eqref{eq:fg-relation} named $\sim$,
 is Noetherian and locally confluent.
%
\end{lemma}
\begin{proof}
 Noetherian is trivial since each step decreases the length of the list.
 Local confluence is shown with a standard critical pair analysis.
 Assume that we have a list $\ell$ which contains two redexes $(x,x^{-1})$ and $(y,y^{-1})$.
 We write $\ell_x$ for the list with the first redex removed and $\ell_y$ for the list with the second redex removed.
 There are only three cases:
 \begin{enumerate}
  \item The two redexes are the same (they ``fully overlap''). In this case, there is nothing to do: The extended cospan is empty.
  \item The two redexes partially overlap, in the sense that $x^{-1} = y$ (or $y^{-1} = x$, which is equivalent).
  In this case, we again have $\ell_x = \ell_y$.
  \item There is no overlap between the two redexes. We can then remove the redex $(y,y^{-1})$ from $\ell_x$ and have constructed a list equal to the one we get if we remove $(x,x^{-1})$ from $\ell_y$.
 \end{enumerate}
\end{proof}

\begin{corollary}[of \cref{lem:acc-trans-acc}] \label{cor:+-Noetherian}
 If $<$ is well-founded [Noetherian], then so is $<^+$.
\end{corollary}

It is a classic result by by Newman \cite{newman1942theories} that local confluence is as strong as general confluence in the context of Noetherian relations.
The proof by Huet  \cite{Huet1980ConfluentRA} can directly be expressed in our setting.
Strictly speaking, our formulation could be described as stronger since we do not assume that the relation is decidable or even given as a function, but Huet's argument stays exactly the same:
\begin{theorem}[Newman \cite{newman1942theories}, Huet \cite{Huet1980ConfluentRA}] \label{thm:NewmanHuet}
 A locally confluent and Noetherian relation $\leadsto$ is confluent.
\end{theorem}
\begin{proof}
 By Noetherian induction, with $P(u)$ expressing that any extended span ${s \toleadsstar \cdot \leadsto^* t}$ can be completed to a quadrangle.
 If either leg has length $0$, the statement is trivial.
 Otherwise, we can write it as ${s {}^* \!\! \toleads v \toleads u \leadsto w \leadsto^* t}$.
 By local confluence, we get $x$ such that $v \leadsto^* x$ and $w \leadsto^* x$.
 By $P(v)$, we get $y$ with $s \leadsto^* y$ and $x \leadsto y$.
 By $P(w)$, we get $z$ with $y \leadsto^* z$ and $t \leadsto^* z$.
 This proof can be pictured as shown below.
 
 \begin{center}
 \begin{tikzpicture}[x=1.2cm,y=-.75cm] 
  \node (U) at (0,0) {$u$};
  \node (V) at (-1,1) {$v$};
  \node (W) at (1,1) {$w$};
  \node (S) at (-2,2) {$s$};
  \node (T) at (2,2) {$t$};
  \node (X) at (0,2) {$x$};
  \node (Y) at (-1,3) {$y$};
  \node (Z) at (0,4) {$z$};
  
  \draw[->] (U) to node [above,sloped] {} (V);
  \draw[->] (U) to node [above,sloped] {} (W);
  \draw[->] (V) to node [above,sloped] {} (S);
  \draw[->] (V) to node [above,sloped] {} (X);
  \draw[->] (W) to node [above,sloped] {} (X);
  \draw[->] (W) to node [above,sloped] {} (T);
  \draw[->] (S) to node [above,sloped] {} (Y);
  \draw[->] (X) to node [above,sloped] {} (Y);
  \draw[->] (Y) to node [above,sloped] {} (Z);
  \draw[->] (T) to node [above,sloped] {} (Z);
  
  \node (star1) at (-1.85,1.7) {\tiny *};
  \node (star2) at (-.15,1.7) {\tiny *};
  \node (star3) at (.15,1.7) {\tiny *};
  \node (star4) at (1.85,1.7) {\tiny *};
  \node (star5) at (-1.15,2.7) {\tiny *};
  \node (star6) at (-.85,2.7) {\tiny *};
  \node (star7) at (-.15,3.7) {\tiny *};
  \node (star6) at (.15,3.7) {\tiny *};
  
  \node (label1) at (0,1) {\small (l.\ confl.)};
  \node (label2) at (-1,2) {\small $P(v)$};
  \node (label1) at (.5,2.5) {\small $P(w)$};
 \end{tikzpicture}
 \end{center}
\end{proof}

\section{Ordering Cycles} \label{sec:order-cycles}

Let again $<$ be a relation on $A$.
The aim of this section is to construct a relation on cycles $\cycle{<}$.
We proceed in several steps.
We start by defining a relation on ordinary lists, i.e.\ on the type $\List \, A$, similar to the well-known \emph{multiset extension}.
Next, we extend the relation to \emph{rotating} lists.
Finally, we construct the relation on cycles and prove essential properties.

\subsection{An Order on Lists}

The type of \emph{lists on $A$} is, as it is standard, inductively generated by a constructor $\mathsf{nil}$ for the empty list and a constructor adding a single element, and it is written as $\List \, A$.
As for chains, we write $\consPrime$ for both adding a single element and for list concatenation.
We write $[a]$ for the list of length $1$ with the single element $a:A$.

\begin{definition} \label{def:list-order}
 Given the relation $<$ as before, we define a relation $<^L$ on $\List \, A$:
 For lists $k, l$, we define $k <^L l$ to mean that $k$ can be constructed from $l$ by replacing a single list element $x$ in $l$ by a finite number of elements of $A$, all of which are smaller than $x$.
\end{definition}
 
 A type-theoretic implementation of this relation is the following.
 For $k : \List \, A$ and $x : A$, we first define $k <^{\mathsf{all}} x$ as ``every element $y$ of $k$ satisfies $y < x$''.
 The relation $<^L$ is then generated inductively by a single constructor as in
 \begin{equation}
  \begin{aligned}
   &\mathsf{inductive } \; (<^L) : \List \, A \to \List \, A \to \UU \\
   & \qquad \mathsf{step}: \Pi\{k, l_1, l_2 : \List \, A\}.\Pi\{x : A\}. \\
   & \hspace{1.5cm} (k <^{\mathsf{all}} x) \to  (l_1 \cons k \cons l_3) <^L (l_1 \cons x \cons l_3)
  \end{aligned}
 \end{equation}


The following strengthening of the induction principle is derivable:

\begin{lemma}[nested induction] \label{lem:nested-ind}
 Assume we are given a relation $<_1$ on a type $B$, a relation $<_2$ on a type $C$, and a family $P : B \times C \to \UU$.
 Assume further that we are given
 \begin{equation}\label{eq:nested-ind-assumption}
 \begin{aligned}
  & \Pi(b : B, c:C). \acc {<_1} (b) \to \acc {<_2} (c) \to \\
  & \qquad (\Pi(b' <_1 b). P(b',c)) \to (\Pi(c'<_2 c). P(b,c')) \to \\
  & \qquad P(b,c).
 \end{aligned}
 \end{equation}
 Then, we get:
 \begin{equation}
  \Pi(b : B,c : C). \acc {<_1} (b) \to \acc {<_2} (c) \to P(b,c).
 \end{equation}
\end{lemma}
\begin{proof}
 By double induction on the accessibility witnesses.
\end{proof}

The next three short lemmas show that $<^L$ is well-founded assuming that $<$ is.
The argument mirrors Nipkov's proof that the multiset extension preserves well-foundedness \cite{nipkow:multiset}.

\begin{lemma}\label{lem:sum-accessible}
 If lists $l$ and $k$ are both $<^L$-accessible, then so is $l \cons k$.
\end{lemma}
\begin{proof}
 We use \cref{lem:nested-ind} with $\List \, A$ and $<^L$ as both $<_1$ and $<_2$, and with $P(l,k) \defeq \acc {<^L} (l \cons k)$.
 We need to show \eqref{eq:nested-ind-assumption}.
 Assume $l$ and $k$ are $<^L$-accessible.
 By definition, $P(l,k)$ holds if every list smaller than $l \cons k$ is accessible.
 By \cref{def:list-order}, a smaller list can either be written as $l' \cons k$ with $l' <^L l$ or as $l \cons k'$ with $k' <^L k$.
 These are accessible by induction hypothesis.
\end{proof}

\begin{lemma} \label{lem:singleton-list-accessible}
 If $a$ is $<$-accessible, then the singleton list $[a]$ is $<^L$-accessible.
\end{lemma}
\begin{proof}
 We do $<$-induction on $a$.
 It is enough to show that $l \equiv [a_1, a_2, \ldots, a_n]$ is $<^L$-accessible for any $l <^L [a]$.
 By the induction hypothesis, every $[a_i]$ is $<^L$-accessible.
 By \cref{lem:sum-accessible}, the whole list is $<^L$-accessible.
\end{proof}

\begin{lemma} \label{thm:ll-wf}
 If $<$ is well-founded, then so is $<^L$.
\end{lemma}
\begin{proof}
 Writing a given list $l$ as the concatenation of singleton lists and applying \cref{lem:singleton-list-accessible}, the list $l$ is $<^L$-accessible.
\end{proof}

\subsection{A Rotation-Invariant Relation on Lists}

\emph{Rotating} a list is defined in the obvious way:
The function
\begin{equation} \label{eq:rot}
 \rotate : \List \, A \to \List \, A
\end{equation}
removes the first element of a list (if it exists) and adds it at the end of the list (left rotation; we do not need right rotation).
Rotating is also known as \emph{cyclic permutation}.
We now define a version of $<^L$ which does not discriminate between lists that are the same module rotation.

\begin{definition}
 Given the relation $<$ on $A$,
 we define the relation $<^{\smallerhexagon}$ on $\List\, A$ 
 by setting
 \begin{equation}
  (k <^{\smallerhexagon} l) \defeq \Sigma (n : \N). \rotate^n(k) <^L l.
 \end{equation}
\end{definition}

By definition, the relation $<^{\smallerhexagon}$ is invariant under arbitrarily rotating the smaller list.
We also have this property for the larger list:
\begin{lemma} \label{lem:rotation-invariant}
 If we have $k <^{\smallerhexagon} l$, then we also have that $k <^{\smallerhexagon} \rotate^m(l)$ for any natural number $m$.
\end{lemma}
\begin{proof}
 It is enough to show the statement for $m \equiv 1$.
 Let us assume that we have a number $n$ such that $\rotate^n(k) <^L l$.
 By definition of $<^L$, the latter expression can be written as 
 \begin{equation}
  (l_1 \cons l_2 \cons l_3) <^L (l_1 \cons x \cons l_3).
 \end{equation}
 If we rotate the right-hand side once, then there are two cases: If $l_1$ is non-empty, then rotating the left-hand side once gives again an instance of $<^L$.
 If $l_1$ is the empty list, then rotating the left-hand side $\mathsf{length}(l_2)$ times yields an instance of $<^L$.
\end{proof}

\begin{corollary}\label{cor:l-rotl-same-predecessors}
 A list $l$ is $<^{\smallerhexagon}$-accessible if and only if $\rotate(l)$ is. \qed
\end{corollary}

\begin{lemma} 
 If a list $l$ is $<^L$-accessible, then it is $<^{\smallerhexagon}$-accessible.
\end{lemma}
\begin{proof}
 We do $<^L$-accessibility induction on $l$ and assume $\acc {<^L} (l)$.
 This allows us to assume that any $k$ with $k <^L l$ is $<^{\smallerhexagon}$-accessible.

 Given any $k'$ with $k' <^{\smallerhexagon} l$, we need to show that $k'$ is $<^{\smallerhexagon}$-accessible.
 By definition, we have a number $n$ such that $\rotate^n(k') <^L l$.
 Therefore, $\rotate^n(k')$ and by \cref{cor:l-rotl-same-predecessors} is $<^{\smallerhexagon}$-accessible.
\end{proof}

\begin{corollary} \label{cor:llrot-wf}
 If the relation $<^L$ is well-founded, then so is $<^{\smallerhexagon}$.
\end{corollary}

\subsection{An Order on Cycles}

Rotating a chain is defined analogously to rotating a list. Specialising to cycles and re-using the name introduced at \eqref{eq:rot}, this gives us a function
\begin{equation}
 \rotate : \cycle < \to \cycle <.
\end{equation}
This function essentially rotates the base point of the cycle.

For a given chain, we can produce the list of its vertices and forget the actual relation.
We define this function as follows:
\begin{equation}
 \begin{aligned}
 & \pi : (a \leadsto^{s*} b) \to \List \, A \\
 & \pi (\mathsf{nil}_a) \defeq \mathsf{nil} \label{eq:pi-nil} \\
 & \pi (\consLong_{a,b,c}(f,p)) \defeq \pi(f) \cons c
 \end{aligned}
\end{equation}
Specialised to cycles, this function
\begin{equation} 
\pi : \cycle < \to \List \, A
\end{equation}
lists the vertices of a cycle.
Clearly, the functions $\rotate$ and $\pi$ commute.
Note that the function $\pi$ ignores the very first endpoint of a chain: $\pi (\mathsf{nil}_a)$ is the empty list, \emph{not} the list containing $a$.
This ensures that $\pi$ works on cycles as expected and does not list the base point twice.
(It also means that an empty cycle, although still being based at some point, has an empty list of vertices.
This awkward edge case will not be relevant.)

We can now extend the relation $<^L$ to chains and $<^{\smallerhexagon}$ to cycles, overloading symbols:
 
\begin{definition}[inherited relations on chains and cycles] \label{def:inh-relation}
 For chains $f : a <^{s*} b$ and $g : c <^{s*} d$, we define $f <^L g$ to mean $\pi(f) <^L \pi(g)$.
 For cycles $p,q : \cycle{<}$, we define $p <^{\smallerhexagon} q$ to mean $\pi(p) <^{\smallerhexagon} \pi(q)$.
\end{definition}

\begin{theorem} \label{thm:Noeth-cycle-rel}
 If $\leadsto$ is Noetherian on $A$, then $\leadsto^{+ \smallerhexagon +}$ is Noetherian on $\cycle{\leadsto}$.
\end{theorem}
\begin{proof}
 By combining \cref{thm:ll-wf} with \cref{cor:+-Noetherian,cor:llrot-wf}.
\end{proof}

\subsection{Constructions with Noetherian Relations on Cycles}

\begin{definition}[span]
 Let $\gamma : \cycle \leadsto$ be a cycle.
 We say that $\gamma$ \emph{contains a span} if we find consecutive vertices $v \toleads u \leadsto w$, i.e.\ if we have
 \begin{equation} \label{eq:span-filter}
  \gamma = (\alpha \cons s^{-1} \cons t \cons \beta)
 \end{equation}
 for some $\alpha : a \leadsto^{s*} v$, $s : u \leadsto v$, $t : u \leadsto w$, $\beta : w \leadsto^{s*} a$.
\end{definition}

The very intuitive statement of the next lemma is illustrated in \cref{fig:cospan-cycle} below.
\begin{lemma} \label{lem:span-filter}
 If a relation $\leadsto$ is well-founded (or Noetherian), then any cycle $\gamma : \cycle \leadsto$ is either empty or contains a span.
\end{lemma}
\begin{proof}
 If $\alpha$ is empty, we are done.
 Otherwise, by \cref{lem:acc-trans-acc} combined with \cref{lem:a<a-no-acc}, $\gamma$ is not monotone in the sense of \cref{def:chains-cycles}, i.e.\ 
 not all segments of $\gamma$ go into the same direction.
 This means that we can find a span by going over the vertices of the cycle.
\end{proof}

\begin{figure}[H] 
\newdimen\R
\newdimen\S
\R=1.4cm
\newcommand{\rot}{-180}
\newcommand{\srot}{-180}

\begin{minipage}[t]{0.25\textwidth}
\begin{center}
\begin{tikzpicture}[baseline=(current bounding box.center)]
  \tikzset{arrow/.style={shorten >=0.1cm,shorten <=.1cm,-latex}}

\node (P1) at (\rot-0:\R) {$a_0$}; 
\node (P2) at (\rot-72:\R) {$a_1$}; 
\node (P3) at (\rot-144:\R) {$a_2$}; 
\node (P4) at (\rot-216:\R) {$a_3$}; 
\node (P5) at (\rot-288:\R) {$a_4$}; 

\draw[arrow] (P1) to node [below] {} (P2);
\draw[arrow] (P2) to node [below] {} (P3);
\draw[arrow] (P4) to node [left] {} (P3);
\draw[arrow] (P4) to node [left] {} (P5);
\draw[arrow] (P1) to node [left] {} (P5);
\end{tikzpicture}
\end{center}
\end{minipage}%
\begin{minipage}[t]{0.25\textwidth}
\begin{center}
\begin{tikzpicture}[baseline=(current bounding box.center)]
  \tikzset{arrow/.style={shorten >=0.1cm,shorten <=.1cm,-latex}}

\node (P1) at (\srot-0:\R) {$b_0$}; 
\node (P2) at (\srot-72:\R) {$b_1$}; 
\node (P3) at (\srot-144:\R) {$b_2$}; 
\node (P4) at (\srot-216:\R) {$b_3$}; 
\node (P5) at (\srot-288:\R) {$b_4$}; 

\draw[arrow] (P1) to node [below] {} (P2);
\draw[arrow] (P2) to node [below] {} (P3);
\draw[arrow] (P3) to node [left] {} (P4);
\draw[arrow] (P4) to node [left] {} (P5);
\draw[arrow] (P5) to node [left] {} (P1);
\end{tikzpicture}
\end{center}
\end{minipage}

\caption{Two elements of $\cycle <$.
The left cycle consists of points $a_i : A$ and contains two spans, namely ${a_4 \protect\toleads a_0 \leadsto a_1}$ and $a_2 \protect\toleads a_3 \leadsto a_4$.
The right cycle does not contain any span; this is not possible if the relation is Noetherian.%
}\label{fig:cospan-cycle}
\end{figure}
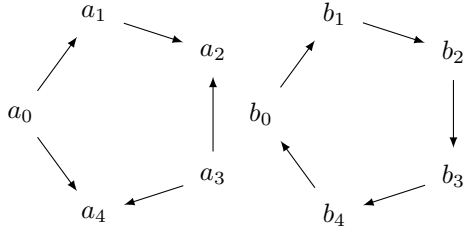

\begin{definition}[merge of cycles] \label{def:merge}
 Assume we have $a,b : A$ and $\alpha, \beta, \gamma : a \sim^{*} b$.
 We say that the cycle $\alpha \cons \gamma^{-1}$ is the \emph{merge} of the cycles
 $\alpha \cons \beta^{-1}$ and $\beta \cons \gamma^{-1}$.
\end{definition}

Merging is also known as \emph{pasting} \cite{vincent-16} or \emph{composition} \cite{vincent-08}.


\begin{figure}[h] 
\begin{center}
\begin{tikzpicture}[x=.6cm,y=-.6cm,baseline=(current bounding box.center)]
  \tikzset{arrow/.style={shorten >=0.1cm,shorten <=.1cm,-latex}}

\node (P1) at (0,0) {$a$}; 
\node (P2) at (2,1) {$a_1$}; 
\node (P3) at (4,0) {$a_2$}; 
\node (P4) at (6,-1) {$a_3$}; 
\node (P5) at (8,0) {$b$}; 

\node (Q1) at (-1,-3) {$b_1$}; 
\node (Q2) at (1,-5) {$b_2$}; 
\node (Q3) at (3,-4) {$b_3$}; 
\node (Q4) at (6,-4) {$b_4$}; 
\node (Q5) at (9,-2) {$b_5$}; 

\node (R1) at (0,3) {$c_1$}; 
\node (R2) at (2,4) {$c_2$}; 
\node (R3) at (5,4) {$c_3$}; 
\node (R4) at (7,2) {$c_4$}; 

\node (LABEL1) at (4.5,-4.5) {$\alpha$}; 
\node (LABEL2) at (3,-.25) {$\beta$}; 
\node (LABEL3) at (3.5,3.5) {$\gamma$}; 

\node (LABEL4) at (2.5,-2) {$\alpha \cons \beta^{-1}$}; 
\node (LABEL5) at (4.5,2) {$\beta \cons \gamma^{-1}$};

\draw[arrow] (P1) to node [below] {} (P2);
\draw[arrow] (P2) to node [below] {} (P3);
\draw[arrow] (P4) to node [left] {} (P3);
\draw[arrow] (P4) to node [left] {} (P5);

\draw[arrow] (Q1) to node [below] {} (P1);
\draw[arrow] (Q1) to node [below] {} (Q2);
\draw[arrow] (Q3) to node [below] {} (Q2);
\draw[arrow] (Q3) to node [below] {} (Q4);
\draw[arrow] (Q4) to node [below] {} (Q5);
\draw[arrow] (Q5) to node [below] {} (P5);

\draw[arrow] (P1) to node [below] {} (R1);
\draw[arrow] (R1) to node [below] {} (R2);
\draw[arrow] (R2) to node [below] {} (R3);
\draw[arrow] (R4) to node [below] {} (R3);
\draw[arrow] (P5) to node [below] {} (R4);

\end{tikzpicture}
\end{center}
\caption{The \emph{merge} of cycles as in \cref{def:merge}.
The big cycle is the merge of the two smaller ones.
Technically, any of the three cycles is the merge of the other two.
}\label{fig:merge}
\end{figure}
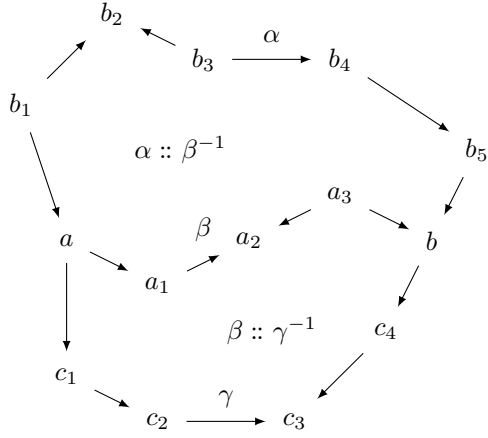

Let $\leadsto$ be locally confluent.
By \cref{def:confluence}, we get a map from spans to extended cospans.
By concatenating the original span with (the reversed version of) this extended cospan, we get a cycle.
We denote this function from spans to cycles by
\begin{equation} \label{eq:spantocycle}
 \spantocycle: (\cdot \toleads \cdot \leadsto \cdot) \to \cycle \leadsto
\end{equation}
and call a cycle of the form $\spantocycle(\kappa)$ a \emph{confluence cycle}.%
\footnote{The more precise but also longer description would be \emph{local confluence cycle}, or \emph{local confluence diagram}.}

\begin{lemma} \label{lem:write-cycle-as-merge}
 Let $\leadsto$ be Noetherian and locally confluent.
 Any cycle $\gamma$ is either empty or there is a number $n$ such that the rotation $\rotate^n(\gamma)$ 
 can be written as the merge of a confluence cycle and a cycle $\delta$ with $\gamma \leadsto^{+\smallerhexagon+} \delta$.
\end{lemma}
\begin{proof}
 If $\gamma$ is non-empty then, by \cref{lem:span-filter}, it contains a span.
 We can then write $\gamma$ as given by \eqref{eq:span-filter}.
 By rotating, we can bring the cycle into the form $\rotate^n(\gamma) = (\kappa \cons \tau^{-1})$ with $\kappa : v \toleads u \leadsto w$ being the span and $\tau : v \leadsto^{s*} w$ being a chain.
 By local confluence, there is a chain $\alpha : v \leadsto^{s*} w$ which is smaller than the span in the sense that $\kappa \leadsto^{+L+} \alpha$. 
 We define $\delta \defeq \alpha \circ \tau^{-1}$ and have $\rotate^n(\gamma) \leadsto^{+\smallerhexagon+} \delta$, thus by \cref{lem:rotation-invariant} $\alpha \leadsto^{+\smallerhexagon+} \delta$ as required.
\end{proof}

\begin{remark}
 With the construction in the above proof, we have in almost all cases $\gamma \leadsto^{+\smallerhexagon} \delta$.
 Only if the chain $\alpha$ is empty, which may happen if $v = w$, we need two steps from $\gamma$ to $\delta$.
For this reason, the lemma states $\gamma \leadsto^{+\smallerhexagon +} \delta$
\end{remark}

\section{Families over Cycles} \label{sec:fam-of-cycles}

We move on from studying cycles to exploring type families indexed over cycles.

\begin{definition}[stability under merging and rotating]
 Assume that $Q : \cycle \leadsto \to \UU$ is a family of types.
 We say that $Q$ is \emph{stable under merging} if, whenever $\gamma$ is the merge of cycles $\alpha$ and $\beta$ such that we have $Q(\alpha)$ and $Q(\beta)$, we also have $Q(\gamma)$. 
 We say that $Q$ is \emph{stable under rotating} if, for any cycle $\delta$ such that $Q(\delta)$, we also have $Q(\rotate(\delta))$.
\end{definition}

\begin{theorem}[Noetherian cycle induction] \label{thm:cycle-ind}
 Let $\leadsto$ be Noetherian and locally confluent.
 Assume further that $Q : \cycle \leadsto \to \UU$ is stable under merging and rotating.
 If $Q$ is inhabited at every empty cycle and at every confluence cycle, then $Q$ is inhabited everywhere:
 \begin{equation}
  \begin{alignedat}{1}
             & \Pi(\gamma : \biggerhexagon^\leadsto_\emptyset). Q(\gamma) \\
   \to \quad & \Pi(\kappa : \cdot \toleads \cdot \leadsto \cdot). Q(\spantocycle(\kappa)) \\
   \to \quad & \Pi(\alpha : \cycle \leadsto). Q(\alpha).
  \end{alignedat}
 \end{equation}
 If $A$ is a set, then the first line of the above principle can be replaced and it can be stated as follows:
 \begin{equation}
  \begin{alignedat}{1}
   & \Pi(a:A). Q(\epsilon_a) \\
   \to \quad & \Pi(\kappa : \cdot \toleads \cdot \leadsto \cdot). Q(\spantocycle(\kappa)) \\
   \to \quad & \Pi(\alpha : \cycle \leadsto). Q(\alpha).
  \end{alignedat}
 \end{equation}
\end{theorem}
\begin{proof}
 The relation $\leadsto^{+\smallerhexagon+}$ is Noetherian by \cref{thm:Noeth-cycle-rel}.
 We perform Noetherian induction.
 The induction hypothesis in particular tells us that $Q$ is inhabited for the smaller cycle given by \cref{lem:write-cycle-as-merge}.
 From this and from the assumption that $Q$ is inhabited for every confluence cycle, we get the desired result.
 The version for $A$ being a set follows from \cref{rem:empty-versus-trivial}.
\end{proof}

\begin{figure}[H]
\newdimen\R
\newdimen\S
\R=2.4cm
\newcommand{\rot}{-54}

\begin{tikzpicture}[baseline=(current bounding box.center)]
  \tikzset{arrow/.style={shorten >=0.1cm,shorten <=.1cm,-latex}}

\node (P1) at (\rot-0:\R) {$a_0$}; 
\node (P2) at (\rot-45:\R) {$a_1$}; 
\node (P3) at (\rot-90:\R) {$a_2$}; 
\node (P4) at (\rot-135:\R) {$a_3$}; 
\node (P5) at (\rot-180:\R) {$a_4$}; 
\node (P6) at (\rot-225:\R) {$a_5$}; 
\node (P7) at (\rot-270:\R) {$a_6$}; 
\node (P8) at (\rot-315:\R) {$a_7$}; 

\draw[arrow] (P1) to node [below] {}  (P2);
\draw[arrow] (P2) to node [below] {}  (P3);
\draw[arrow] (P4) to node [left] {}  (P3);
\draw[arrow] (P4) to node [left] {}  (P5);
\draw[arrow] (P5) to node [above] {}  (P6);
\draw[arrow] (P7) to node [above] {}  (P6);
\draw[arrow] (P7) to node [right] {}  (P8);
\draw[arrow] (P1) to node [right] {} (P8);

\node (P4b) at (\rot-19:0.5cm) {$a_8$};
\node (P5b) at (\rot-175:1.0cm) {$a_9$};
\draw[arrow, dashed] (P3) to node [] {} (P4b);
\draw[arrow, dashed] (P5) to node [] {} (P5b);
\draw[arrow, dashed] (P5b) to node [] {} (P4b);

\node (P6b) at (\rot-255:1cm) {$a_{10}$};
\draw[arrow, dotted] (P5b) to node [] {} (P6b);
\draw[arrow, dotted] (P6) to node [] {} (P6b);
\end{tikzpicture}
\caption{An illustration of Noetherian cycle induction.
 Assume we want to show a property $Q$ for the big octagon $a_0 - a_7$.
 We first ``remove'' the confluence cycle spanned by
 $a_2 \protect\toleads a_3 \leadsto a_4$ to get the nonagon consisting of $a_0 - a_9$ without $a_3$ but with the dashed edges: we now have more vertices, but the nonagon is smaller than the octagon in the order $\leadsto^{+\protect\smallerhexagon +}$.
 In the next step, we ``remove'' the confluence cycle spanned by $a_9 \protect\toleads a_4 \leadsto a_5$, and so on.
 Which confluence cycle is removed in a certain step depends on the precise proof of \cref{lem:span-filter}.%
}\label{fig:cycle-induction}
\end{figure}
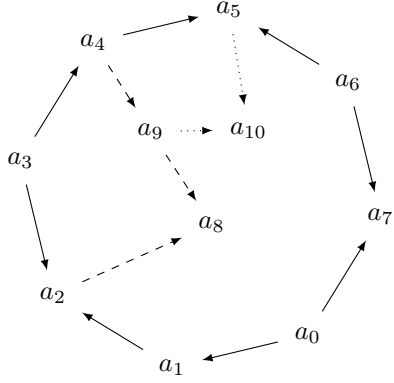

\section{Applications of Noetherian Cycle Induction} \label{sec:apps}

We are now ready to combine the theory developed in the previous two sections with our result on set-quotients.

\begin{theorem} \label{thm:dirsetquot}
 Let $A : \UU$ be a type, $\leadsto : A \to A \to \UU$ be a Noetherian and locally confluent relation, and $X : \UU$ be a 1-type.
 Then, the type of functions 
 $(A \slash \leadsto) \to X$
 is equivalent to the type of tuples $(f,h,d_1,d_2)$ (a nested $\Sigma$-type), where
 \begin{align}
  & f : A \to X \\
  & h : \Pi\{a, b : A\}. (a \leadsto b) \to f(a) = f(b) \\
  & d_1 : \Pi\{a : A\}. \Pi(p : a = a). \ap_f(p) = \refl \label{eq:dir-penu-comp} \\
  & d_2 : \Pi (\kappa : \cdot \toleads \cdot \leadsto \cdot). h^{s*} (\spantocycle(\kappa)) = \refl. \label{eq:dir-last-comp}
\end{align}
Further, if $A$ is a set, the type $(A \slash \leadsto) \to X$ is equivalent to the type of triples $(f,h,d_2)$.
\end{theorem}
\begin{proof}
 The case for $A$ being a set follows immediately from the main statement, since the type of $d_1$ becomes contractible.

 For the main statement,
 we want to apply \cref{thm:gensetquotnew}.
 We need to show that the type of $c$ in \eqref{eq:genset-last-comp} is equivalent to the type of pairs $(d_1,d_2)$ above.
 Note that they are all propositions.
 From $c$, we immediately derive $(d_1, d_2)$ from \cref{rem:empty-versus-trivial}.
 
 Let us assume we are given $(d_1, d_2)$. We need to derive $c$.
 We want to apply Noetherian cycle induction as given by \cref{thm:cycle-ind} with $Q(\gamma) \defeq h^{s*}(\gamma) = \refl$.
 First of all, we need to check that $Q$ is stable under merging and rotating.
 Rotating is easy and follows from the fact that left- and right-inverses of path composition coincide.
 Merging follows from associativity of path composition and (\ref{eq:star-inverting}, \ref{eq:star-functoriality}).
 We refer to our Lean formalisation for the details.
 
 Finally, we check the two assumption of \cref{thm:cycle-ind}.
 Let $\gamma$ be an empty cycle. With \cref{rem:empty-versus-trivial}, we can construct $Q(\gamma)$ from $d_1$.
 Let $\kappa$ be a span. What we need is given by $d_2$.
\end{proof}

We want to use this theorem to show that the free higher group $F_A$ has trivial fundamental groups.
Recall that this is the example discussed in the introduction, with $F_A$ defined in equation \eqref{eq:FA-definition}.

\begin{theorem}\label{thm:fundamental-group-trivial}
 The fundamental groups of the free higher group on a set are trivial.
 In other words, for a set $A$ and any $x : F_A$,
 we have
 \begin{equation}
  \pi\left(F_A , x \right) \; = \; \unit.
 \end{equation}
\end{theorem}

We split the proof into several small lemmas.
We keep using the relation $\leadsto$ of \cref{ex:fg,lem:fg-Noether-confl}.
Further, recall 
the functions $\omega_1$ \eqref{eq:omega1} and $\omega_2$ \eqref{eq:omega2} from the introduction,
as well as 
the map $\omega$ \eqref{eq:omega-complete}.

\begin{lemma}[{free group; continuing \cref{ex:fg,lem:fg-Noether-confl}}] \label{lem:fg-cycle-condition}
 For the local confluence constructed in \cref{lem:fg-Noether-confl}, we can construct a term
 \begin{equation}
  d_2 : \Pi(\kappa : \cdot \toleads \cdot \leadsto \cdot). \omega_2^{s*}(\spantocycle(\kappa)) = \refl.
 \end{equation}
\end{lemma}
\begin{proof}
%
 This lemma allows a proof of the (for HoTT typical) form "generalise and do path induction".
 Unfortunately, it is somewhat tedious.
 The function $\omega_1 : \List(A+A) \to F_A$, cf.\ \eqref{eq:omega1}, factors as
 \begin{equation} \label{eq:factor-omega1}
  \List(A+A) \longrightarrow \List(\mathsf{base} = \mathsf{base}) \longrightarrow \mathsf{base} = \mathsf{base}
 \end{equation}
 where the first map applies $\mathsf{loop}$ on every list element, while the second concatenates; note that $F_A \defeq (\mathsf{base} = \mathsf{base})$.
 The function $\omega_2$ \eqref{eq:omega2} can then be factored similarly.

 We treat the three cases in \cref{lem:fg-Noether-confl} separately. 
 The third case is the hardest:
 We need to show that, for a given list of the form $[\ldots, x, x^{-1}, \ldots, y, y^{-1}, \ldots]$, the two ways of simplifying the list get mapped to equal 2-paths between parallel 1-paths.
 Using \eqref{eq:factor-omega1}, we can instead assume that we are given a list of loops around $\mathsf{base}$.
 We first repeatedly use that associativity of path composition is coherent (we have ``MacLane's pentagon'' by trivial path induction).
 Then, we have to show that the two canonical ways of simplifying $e_1 \ct (p \ct p^{-1}) \ct e_2 \ct (q \ct q^{-1}) \ct e_3$ to $e_1 \ct e_2 \ct e_3$ are equal.
 The usual strategy in HoTT is now to generalise to the case that $p$ and $q$ are equalities with arbitrary endpoints rather than loops, and then do path induction.
If $p$ and $q$ are both $\refl$, then both simplifications become $\refl$ as well.
Instead of path induction, it seems possible to prove this lemma only using naturality and Eckmann-Hilton~\cite[Thm 2.6.1]{hott-book}.
 
 In the second case of \cref{lem:fg-Noether-confl}, we have to show that the two ways of reducing $e_1 \ct p \ct p^{-1} \ct p \ct e_2$ to $e_1 \ct p \ct e_2$ are equal.
 Again, if $p$ is $\refl$, this is automatic.
 The first case is trivial.
\end{proof}

\begin{lemma} \label{lem:retract}
 The free higher group $F_A$ is a retract of $\List(A+A) \sslash \leadsto$, in the sense that there is a map
 \begin{equation}
  \varphi : F_A \to (\List(A+A) \sslash \leadsto)
 \end{equation}
 such that $\omega \circ \varphi$ is the identity on $F_A$.
\end{lemma}
\begin{proof}
 For any $x : A+A$, the operation ``adding $x$ to a list''
 \begin{equation}
  (x \cons \_) \; : \; \List(A+A) \to \List(A+A)
 \end{equation}
 can be lifted to a function of type
 \begin{equation} \label{eq:endoeqv-on-lists}
  (\List(A+A) \sslash \leadsto) \to (\List(A+A) \sslash \leadsto).
 \end{equation}
 Moreover, the function \eqref{eq:endoeqv-on-lists} is inverse to $(x^{-1} \cons \_)$ and thus an equivalence.

 Let $\star$ be the unique element of the unit type $\unit$.
 We define the relation $\sim$ on the unit type by $(\star \sim \star) \defeq A$.
 Then, $\hcolim (A \rightrightarrows \unit)$ is by definition the coequaliser $(\unit \sslash \sim)$, and $F_A$ is given by $(\iota(\star) = \iota(\star))$.
 This allows us to define $\varphi$ using \cref{thm:lics2019-main} with the constant family $P \defeq (\List(A+A) \sslash \leadsto)$, with the equivalence of the component $e$ given by \eqref{eq:endoeqv-on-lists}.
 
 A further application of \cref{thm:lics2019-main} show that $\omega \circ \varphi$ is pointwise equal to the identity.
\end{proof}

\begin{proof}[Proof of \cref{thm:fundamental-group-trivial}]
 By \cite[Thms 7.2.9 and 7.3.12]{hott-book}, the statement of the theorem is equivalent to the claim that 
 $\trunc 1 {F_A}$
 is a set.
 
 We now consider the following diagram:
 
\begin{equation} \label{eq:diagram}
\begin{tikzpicture}[x=2.7cm,y=-2cm,baseline=(current bounding box.center)]
  \tikzset{arrow/.style={shorten >=0.1cm,shorten <=.1cm,-latex}}
\node (FA1) at (0,0) {$F_A$}; 
\node (Lquot) at (0,1) {$\List(A+A) \sslash \leadsto$}; 
\node (Lsquot) at (0,2) {$\List(A+A) \slash \leadsto$}; 
\node (FA2) at (1,1) {$F_A$}; 
\node (FAT) at (2,1) {$\trunc 1 {F_A}$}; 

\draw[arrow] (FA1) to node [right] {$\varphi$} (Lquot);
\draw[arrow] (FA1) to node [above right] {$\tproj 1 -$} (FAT);
\draw[arrow] (Lquot) to node [right] {$\tproj 0 -$} (Lsquot);
\draw[arrow] (Lquot) to node [above] {$\omega$} (FA2);
\draw[arrow] (FA2) to node [above] {$\tproj 1 -$} (FAT);
\draw[arrow,dashed] (Lsquot) to node [below right] {} (FAT);
\end{tikzpicture}
\end{equation}
 The dashed map exists by 
 the combination of \cref{thm:dirsetquot} (note that we are in the simplified case where the type to be quotiented is a set) together with \cref{lem:fg-cycle-condition} (and \cref{lem:simu-is-consec}).
 By construction, the bottom triangle commutes.
 The top triangle commutes by \cref{lem:retract}.
 
 Therefore, the map $\tproj 1 -$ factors through a set (namely $\List(A+A) \slash \leadsto$).
 This means that $\trunc 1 {F_A}$ is a retract of a set, and therefore itself a set.
\end{proof}

There are a number of deep problems that we can approach in a similar fashion.
%
Let us look at the following list, where the first question is the one discussed in the introduction:
\begin{enumerate}[(i)]
\item \label{item:1}
Is the free higher group on a set again a set?
\item \label{item:2}
Is the suspension of a set a $1$-type (open problem recorded in \cite[Ex 8.2]{hott-book})?
\item \label{item:3}
Given a $1$-type $B$ with a base point $b_0 : B$. If we add a single loop around $b_0$, it the type still a $1$-type?
\item \label{item:4}
Given $B$ and $b_0$ as above, imagine we add $M$-many loops around $b_0$ for some given set $M$. Is the resulting type still a $1$-type?
\item \label{item:5}
If we add a path (not necessarily a loop) to a $1$-type $B$, is the result still a $1$-type?
\item \label{item:6}
If we add an $M$-indexed family of paths to a $1$-type $B$ (for some set $M$), is the resulting type still a $1$-type?
\end{enumerate}
All questions are of the form:
\begin{quote}
 ``Can a change at level $1$
   induce a change at level $2$ or higher?''
\end{quote}
Only \ref{item:1} seems to be about level $0$ and $1$, but this is simply because we have taken a loop space.
With our \cref{thm:dirsetquot}, we can show an approximation for each of these questions analogously to \cref{thm:fundamental-group-trivial}.
This means that we show:
\begin{quote}
 ``A change at level $1$ does not induce a change at level $2$ (but we don't know about higher levels).''
\end{quote}

Fortunately, we do not have to discuss each problem separately as they have a common generalisation. 
To see this, assume we have a span $B \xleftarrow f A \xrightarrow g C$ of types and functions.
We then get the following pushout:
  \begin{equation} \label{eq:pushout}
  \begin{tikzpicture}[x=1.5cm,y=-1.0cm,baseline=(current bounding box.center)]
   \node (A) at (0,0) {$A$};
   \node (C) at (1,0) {$C$};
   \node (B) at (0,1) {$B$};
   \node (D) at (1,1) {$B +_A C$};
  
   \draw[->] (A) to node [left] {\scriptsize $f$} (B);
   \draw[->] (A) to node [above] {\scriptsize $g$} (C);
   \draw[->] (B) to node [above] {\scriptsize $i_0$} (D);
   \draw[->] (C) to node [left] {\scriptsize $i_1$} (D);
  \end{tikzpicture}
 \end{equation}
 Assume now that $A$ is a set while $B$ and $C$ are $1$-types. 
The above questions all ask whether $B +_A C$ is a $1$-type,
under the following additional assumptions:
\begin{enumerate}[(i)]
 \item Assuming that $B$, $C$ are both the unit type $\unit$ and $A$ is $A' + \unit$, where $A'$ is the set on which we want the free higher group (this is the usual translation from coequalisers to pushouts); 
 \item Assuming that $B$ and $C$ are both $\unit$;
 \item assuming that $A$ is $\unit$ and $C$ is the circle $\mathsf{S}^1$;
 \item assuming that $A$ is $\unit$ and $C$ is $M \times \mathsf{S}^1$;
 \item assuming that $A$ is the 2-element type $\bool$ and $C$ is $\unit$;
 \item assuming that $A$ is $M \times \bool$ and $C$ is $M$.
\end{enumerate}

With our tools, we can show an approximation of the general problem (and thus approximations of \ref{item:1} to \ref{item:6}):

\begin{theorem} \label{thm:pushout-1-type}
 Given a pushout as in \eqref{eq:pushout}, if $A$ is a set and $B$, $C$ are $1$-types, then all second homotopy groups of $B+_A C$ are trivial.
 In other words, $\trunc 2 {B+_A C}$ is a $1$-type.
\end{theorem}
\begin{proof}[Proof sketch]
The argument is almost completely analogous to the proof of \cref{thm:fundamental-group-trivial}.
The main difference is that the type $\List(A+A)$ is not sufficient any more.
Instead, we need to be slightly more subtle when we encode the equalities in the pushout.
The following construction is due to Favonia and Shulman \cite{favonia:SvK}, who use it in their formulation of the \emph{Seifert-van Kampen Theorem}.

Given the square in \eqref{eq:pushout} and $b, b' : B$, we consider the type $L_{b,b'}$ of lists of the form
 \begin{equation} \label{eq:SvK-lists}
  [b, p_0, x_1, q_1, y_1, p_1, x_2, q_2, y_2, \ldots, y_n, p_n, b'] 
 \end{equation}
 where\footnote{We remove the $0$-truncations around the path spaces. These are without effect here since $B$, $C$ are $1$-types.}
 \begin{alignat}{3}
  & x_i : A & \qquad &  y_i : A \\
  & p_0 : b = f(x_1) && p_n : f(y_n) = b' \\
  & p_i : f(y_i) = f(x_{i+1}) && q_i : g(x_i) = g(y_i)
 \end{alignat}
The corresponding relation is generated by
 \begin{equation} \label{eq:SvK-relation}
 \begin{aligned}
  [\ldots, q_k, y_k, \refl, y_k, q_{k+1}, \ldots] \; & \leadsto \; [\ldots, q_k \ct q_{k+1}] \\
  [\ldots, p_k, x_k, \refl, x_k, p_{k+1}, \ldots] \; & \leadsto \; [\ldots, p_k \ct p_{k+1}] 
 \end{aligned}
 \end{equation}
The statement of the Seifert-van Kampen Theorem is that the set-quotient $L_{b,b'} \slash \leadsto$ is equivalent to the set-truncated type $\trunc 0 {i_0(b) = i_0(b')}$ of equalities in the pushout.
Similarly to $L_{b,b'}$, there are three further types of lists where one or both of the endpoints are in $C$ instead of $B$.
In general, we can define a type of lists $L_{x,x'}$ for $x,x' : B + C$, and the Seifert-van Kampen Theorem states that $L_{x,x'} \slash \leadsto$ is equivalent to $\trunc 0 {i(x) = i(x')}$, with $i : B + C \to B +_A C$ given by $(i_0, i_1)$.

Let us compare the lists of the type \eqref{eq:SvK-lists} with the lists $\List(A+A)$ that we have used before for special case \ref{item:1}.
In this special case, the elements $p_i$ and $q_i$ in \eqref{eq:SvK-lists} carry no information, and the right summand of $A' + \unit$ replaces the additional choice of ``left or right'' in $\List(A+A)$.
Under this transformation, the relations \eqref{eq:fg-relation} and \eqref{eq:SvK-relation} correspond to each other.

The proof of \cref{thm:pushout-1-type} then proceeds as follows.
The construction of $\omega$ and $\varphi$ is essentially the same as before, using the version of \cref{thm:lics2019-main} for pushouts available in \cite{KrausVonRaumer_pathSpaces}.
For the relation \eqref{eq:SvK-relation}, we can show the analogous to \cref{lem:fg-Noether-confl,lem:fg-cycle-condition}.
The analogous to \eqref{eq:diagram} is
\begin{equation} \label{eq:diagram2}
\begin{tikzpicture}[x=2.7cm,y=-2cm,baseline=(current bounding box.center)]
  \tikzset{arrow/.style={shorten >=0.1cm,shorten <=.1cm,-latex}}
\node (FA1) at (0,0) {$i(x) = i(x')$}; 
\node (Lquot) at (0,1) {$L_{x,x'} \sslash \leadsto$}; 
\node (Lsquot) at (0,2) {$L_{x,x'} \slash \leadsto$}; 
\node (FA2) at (1,1) {$i(x) = i(x')$}; 
\node (FAT) at (2,1) {$\trunc 1 {i(x) = i(x')}$}; 

\draw[arrow] (FA1) to node [right] {$\varphi$} (Lquot);
\draw[arrow] (FA1) to node [above right] {$\tproj 1 -$} (FAT);
\draw[arrow] (Lquot) to node [right] {$\tproj 0 -$} (Lsquot);
\draw[arrow] (Lquot) to node [above] {$\omega$} (FA2);
\draw[arrow] (FA2) to node [above] {$\tproj 1 -$} (FAT);
\draw[arrow,dashed] (Lsquot) to node [below right] {} (FAT);
\end{tikzpicture}
\end{equation}
There is a small subtlety: Since $A$ is a set and $B$, $C$ are $1$-types, the type of lists $L_{x,x'}$ is a set.
This is important since it allows us (as before) to use the simpler version of \cref{thm:dirsetquot}.
The above diagram shows that $\trunc 1 {i(x) = i(x')}$ is a set.
Choosing $x$ and $x'$ to be identical, this means that $\trunc 1 {\Omega(B+_A C,i(x))}$ is a set, which is equivalent to the statement that $\trunc 0 {\Omega^2(B+_A C,i(x))}$ (the second homotopy group) is trivial.
It follows by the usual induction principle of the pushout that $\trunc 0 {\Omega^2(B+_A C,z)}$ for arbitrary $z : B+_A C$ is trivial.
\end{proof}

\section{On Type Theory in Type Theory} \label{sec:tt-in-tt}

 There is a long series of suggestions to formalise dependent type theory inside dependent type theory, starting with Dybjer~\cite{dybjer1995internal}.
 Chapman~\cite{chapman2009type} introduced the expression ``type theory eats itself''
 for a formalisation of the type-theoretic syntax, inspired by Danielsson~\cite{danielsson2006formalisation} and others.
 Altenkirch and Kaposi~\cite{alt-kap:tt-in-tt} refer to the same concept as ``type theory in type theory''.
 Another recent suggestion was made by Escard{\'o} and Xu~\cite{autophagia}. 
 Abel, \"Ohmann, and Vezzosi~\cite{abel2017decidability} start their paper stating that ``type theory should be able to handle its own meta-theory''. 

 All proposals construct, in one way or another, types and type families of contexts, substitutions, types, and terms.
 They further define convertibility relations which capture the intended equalities, stating for example that $(\lambda x.x)(0)$ is related to $0$.
 It is then natural to quotient the constructed types by the relation.
Some authors who use homotopy type theory do this without truncating, which leads to a syntax that is not a set and therefore cannot have decidable equality.
Other authors ensure that the constructed syntax is set-truncated and has decidable equality~\cite{kaposi2017normalisation}.
 However, this approach leads to another problem which was discussed by Shulman~\cite{shulman:eating},  Altenkirch and Kaposi~\cite{alt-kap:tt-in-tt}, Kaposi and Kov{\'a}cs \cite{AAhiits}, and others in the community.
 The problem is that we want to ensure that the \emph{standard model} or \emph{meta-circular interpretation} is really a model of the syntax. This means that, among other components, we have to construct a function from the set of contexts to the universe of sets (or types). But this universe is not a set.

If we are happy to build the standard model out of sets (meaning that contexts are sets and types are families of sets), then it is a $1$-type.
This makes the results of this paper relevant.
As noted for example in the Agda implementation by Escard{\'o} and Xu~\cite{autophagia}, all required equations hold \emph{judgmentally} in the standard model. In other words, every $p : a \sim b$ gets interpreted as $\refl$.
But if every segment of a cycle gets mapped to $\refl$, one might expect that it is easy to prove that the whole cycle gets mapped to $\refl$ as required by \cref{thm:gensetquotnew}.
Unfortunately, this is not the case. The first statement is meta-theoretic, while the second is internal.
Although trivial for any \emph{concretely given} cycle, it is unclear to us how to directly prove that \emph{all} cycles get mapped to $\refl$.

While \cref{thm:gensetquotnew} seems to not be useful for this reason, we conjecture that a version of Noetherian cycle induction (\cref{thm:cycle-ind}) can be used. 
The relations we need in type theory are Noetherian and confluent, and it is very plausible that confluence cycles can be checked.
The main difficulty which needs to be overcome is that the authors \cite{alt-kap:tt-in-tt} do not simply first define types of expressions and then the relation, but both components mutually.
It seems necessary to adapt our main result accordingly.
We have not worked this out and leave it for the future.

\section{Final Remarks} \label{sec:concl}

Coherence of structures is a central concept in 
homotopy type theory.
In most cases, the difficult part is to \emph{express} 
coherence.
In the present paper, this part is easy: The coherence in question is already expressed correctly in the fairly simple \cref{thm:gensetquotnew}.
Instead, \emph{proving} or \emph{checking} the coherence of data that occurred in natural examples is the tricky part.


It is natural to ask what one can say about the higher homotopy groups of the free higher group over a set.
It may very well be possible to show ``higher'' versions of \cref{thm:dirsetquot} which relax the condition of $1$-truncatedness if $X$ to, for example, $2$-truncatedness.
The expectation is that this would require coherence for the proofs that one needs for confluence cycles.

We do not believe that the conditions on $A$, $B$, and $C$ in \cref{thm:pushout-1-type} can be relaxed any further.
If $A$ is the non-set $\mathsf{S}^1$ we get a counterexample, since the suspension of the circle is the sphere which does not have trivial higher homotopy groups.
If $B$ and $C$ are not required to be $1$-truncated, Paolo Capriotti has pointed out to us that the pushout of the span
\begin{equation}
 \trunc 2 {\mathsf{S}^2} \leftarrow \unit \rightarrow \trunc 2 {\mathsf{S}^2}
\end{equation}
is not a $2$-type.
Thus, \cref{thm:pushout-1-type} is in some sense maximally general.

Finally, we hope that our main construction has applications outside of type theory. 
The condition that ``every cycle becomes trivial'' makes sense in graph rewriting, although (as far as we are aware) it has not been considered in the context of well-founded relations.
As Christian Sattler has pointed out, there is a connection to van Kampen pushouts; for example, in the category of sets, a pushout is van Kampen if every cycle is trivial \cite[Def 21]{lowe2010van}.


%% file: quotients_acks.tex
We would like to thank Vincent van Oostrom for in-depth comments on this paper, for very detailed explanations on rewriting, and for exciting ideas that could extend this line of work. 

For helpful comments on this work, we thank the participants of the \emph{FAUM} meeting (Herrsching 2019) and \emph{Types in Munich} (online 2020), as well as the members of the theory groups in Birmingham, Budapest, and Nottingham.
In particular, we acknowledge the remarks by Steve Awodey, Ulrik Buchholtz, Thierry Coquand, Eric Finster, Mart{\'i}n H{\"o}tzel Escard{\'o}, Egbert Rijke, Anders M\"ortberg, and Chuangjie Xu. We are especially grateful for the discussions with Christian Sattler.

This work has been supported by the Royal Society under grant No.~URF\textbackslash R1\textbackslash 191055.